\documentclass[11pt]{article}

\usepackage{amssymb,amsmath,tikz,amsthm}
\usepackage{enumerate}
\usepackage[shortlabels]{enumitem}
\usepackage[all]{xy}
\usepackage{vmargin}

\setmarginsrb{26mm}{16mm}{26mm}{16mm}%
     {9mm}{2.5mm}{9mm}{10mm}
\usepackage{xspace}

\usepackage{hyperref}

\newtheorem{lemma}{Lemma}[section]
\newtheorem{proposition}[lemma]{Proposition}
\newtheorem{theorem}[lemma]{Theorem}
\newtheorem{corollary}[lemma]{Corollary}

\newtheorem{question}[lemma]{Question}

\theoremstyle{definition}
\newtheorem{definition}[lemma]{Definition}
\newtheorem{remark}[lemma]{Remark}
\newtheorem{example}[lemma]{Example}

\def\K{\mathbb K}
\def\N{\mathbb N}
\def\R{\mathbb R}
\def\Z{\mathbb Z}
\def\Q{\mathbb Q}

\def\Aut{\mathrm{Aut}}
\def\Se{\mathfrak S}
\def\X{\mathfrak X}

\newcommand*{\card}[1]{\left\vert #1 \right\vert}

\newcommand{\htop}{h_{top}}
\newcommand{\hm}{h_{mes}}
\newcommand{\halg}{h_{alg}}
\newcommand{\hset}{\texorpdfstring{\ensuremath{h_{set}}}{h-set}}

\def\ent{\mathrm{ent}}
\def\hti{\widetilde{h}}

\newcommand{\F}{\mathcal F}

\newcommand{\LCG}{\textbf{\textup{LC}}}
\newcommand{\LC}{{}_{\K}\LCG}
\newcommand{\LLC}{{}_{\K}\textbf{\textup{LLC}}}
\newcommand{\LCA}{\textbf{\textup{LCA}}}

\newcommand{\CC}{\LCA_{cc}}
\newcommand{\Vect}{{}_{\K}\textbf{\textup{Vect}}}
\newcommand{\Flow}{\mathrm{Flow}}
\newcommand{\bcL}{\mathbf{\mathcal{L}}}
\newcommand{\Lqm}{\bcL_{qm}}
\newcommand{\barLqm}{\overline{\bcL}_{qm}}
\newcommand{\Ab}{\textbf{\textup{Ab}}}
\newcommand{\TDLC}{\textbf{\textup{TDLC}}}

\newcommand{\LCO}{\mathcal{LCO}}
\newcommand{\CO}{\mathcal{CO}}

\def\Inv{\mathrm{Inv}}

\numberwithin{equation}{section}

\newlength{\bibitemsep}
\newlength{\bibparskip}
\let\oldthebibliography\thebibliography
\renewcommand\thebibliography[1]{%
  \oldthebibliography{#1}%
  \setlength{\parskip}{\bibitemsep}%
  \setlength{\itemsep}{\bibparskip}%
}

\author{Ilaria Castellano \and Dikran Dikranjan \and Domenico Freni \and Anna Giordano Bruno \and Daniele Toller}
\title{Intrinsic entropy for generalized quasimetric semilattices}
\date{Dedicated to the memory of Silvana Rinauro}

\begin{document}

\maketitle

\abstract{We introduce the notion of intrinsic semilattice entropy $\hti$ in the category $\Lqm$ of generalized quasimetric semilattices and contractive homomorphisms. By using appropriate categories $\mathfrak X$ and functors $F:\mathfrak X\to\Lqm$, we find specific known entropies $\hti_\X$ on $\X$ as intrinsic functorial entropies, that is, as $\hti_\mathfrak X=\hti\circ F$. These entropies are the intrinsic algebraic entropy, the algebraic and the topological entropies for locally linearly compact vector spaces, the topological entropy for totally disconnected locally compact groups and the algebraic entropy for compactly covered locally compact abelian groups.}



\section{Introduction}

Entropy has been intensively studied in ergodic theory and topological dynamics since the introduction of the measure entropy $\hm$ and the topological entropy $\htop$ for single selfmaps roughly sixty years ago (see \cite{AKM,B,Kolm,Sinai}). In connection with the topological entropy, the algebraic entropy $\halg$ of group endomorphisms was introduced somewhat later (see \cite{AKM,DG,GBSp,P, Pet1,W}), and the adjoint algebraic entropy  $\halg^*$ more recently (see \cite{DGS,GK}). Moreover, the set-theoretic entropy $\hset$ of selfmaps of a set provided with no further structure was defined in \cite{AZD} (see also \cite{DGV,G0,GBVstr}), and used for computing the topological entropy of generalized shifts. For the details about the origin of all these entropies as well as the connections between them, see the surveys \cite{DG-islam,DSV}.

\smallskip
In the presence of such a wealth of entropies, it gradually became clear that a common approach covering all (or at least, most) of them could be very helpful. Such a  common approach  was proposed in \cite{D} aiming at a uniform argument for the basic properties of the above mentioned entropies.  This argument was elaborated, partially in collaboration with Simone Virili, in full detail and proofs in \cite{DGB_PC,DGBuatc}. 

Recall that  an entropy over a category $\X$  is an invariant $h_\X\colon \Flow_\X\to \R_{\geq0}\cup\{\infty\}$ of the category $ \Flow_\X$ of all flows of $\X$: a flow of $\X$ is a pair $(X,\phi)$ consisting of an object $X$ of $\mathfrak X$ and an endomorphism $\phi: X\to X$, whereas a morphism between flows, say $(X,\phi)$ and $(Y,\psi)$, is given by a morphism $\alpha: X \to Y$ of $\mathfrak X$ such that $\alpha\circ\phi=\psi\circ\alpha$.  Usually, one denotes $h_\X(X,\phi)$ simply by $h_\X(\phi)$ for 
 a flow $(X,\phi)$ of $\X$. 

The main idea of the unifying approach from \cite{DGB_PC,DGBuatc} was to define the \emph{semigroup entropy} $h_\Se\colon\Flow_\Se\to\R_{\geq0}\cup\{\infty\}$, where $\Se$ is the category of normed semigroups $(S,v)$ whose morphisms are all semigroup homomorphisms that are contractive with respect to the norm.  
In this way, whenever a category $\X$ allows for a functor $F : \Flow_\X \to \Flow_\Se$, one can obtain an entropy $h_F$ over $\X$ by defining $h_F=h_{\Se}\circ F\colon \Flow_\X\to \R_{\geq0}\cup\{\infty\}$. The entropy $h_F$ was called  \emph{functorial entropy} in \cite{DGBuatc}. As shown in \cite{DGB_PC,DGBuatc}, all entropies listed above (measure entropy, topological entropy, algebraic entropy, adjoint algebraic entropy, set-theoretic entropy)  can be obtained as functorial entropies for appropriate functors $F:\Flow_\X \to \Flow_\Se$, where $\X$ ranges among categories (such as, respectively, the category of measure space, the category of compact spaces, the category of groups and the category of locally compact groups, the category of sets). In all specific cases the functors $F:\Flow_\X \to \Flow_\Se$ are induced from functors $\X\to \Se$ in the obvious way.

Meanwhile, the intrinsic algebraic entropy for endomorphisms of abelian groups was introduced in \cite{DGSV}. Its definition, for a specific endomorphism $\phi:G\to G$ of an abelian group $G$, is based on the subtle notion of \emph{$\phi$-inert subgroup}, inspired by the well-known notion of inert subgroup in the non-abelian context (see \cite{DDR} for further details). 
Later on, the algebraic entropy and the topological entropy of continuous endomorphisms of locally linearly compact vector spaces were defined in \cite{CGBalg,CGBtop}, respectively (see also \cite{Cas19,Cas20}). In these cases, 
the computation of the entropy of an endomorphism $\phi$ depends on the behavior of some subgroups that turn out to be again $\phi$-inert. So in a purely informal way we call those ``intrinsic-like'' entropies. 

Moreover, the general definitions of the topological entropy $\htop$ (see \cite{DSV,GBV}) and the algebraic entropy $\halg$ (see \cite{V}) for locally compact groups, involving Haar measure, are not ``intrinsic'' -- they are covered by a suitable generalization of the scheme in \cite{DGBuatc} with normed semigroups.
Nevertheless,  for totally disconnected locally compact groups and for locally compact strongly compactly covered groups, respectively, $h_{top}$ and $h_{alg}$ allow for an alternative ``intrinsic'' description, which is  handier since it avoids the use of Haar measure, and the limit superior in the general definition becomes a limit (see \cite{GBV,GBST} respectively).

As pointed out in \cite{DGBuatc}, the unifying approach from \cite{DGB_PC,DGBuatc} does not (and cannot) cover these intrinsic-like entropies. 
So, the aim of this paper is to elaborate a common approach to them.
A careful analysis shows that the common feature of all of them is the presence of a semilattice $S$ provided with a kind of ``non-symmetric distance" which may take also value $\infty$, namely a \emph{generalized quasimetric} (rather than a norm as one had so far in \cite{DGBuatc}). We develop the necessary machinery regarding generalized quasimetric semilattices in the forthcoming project \cite{CDFGKT}, starting from the seminal work by Nakamura \cite{Naka} and from similar structures used in topological algebra (see \cite{AD}) and in computer science (see \cite{Sch}). 

Here we introduce and study the notion of {$\phi$-inert element} of a generalized quasimetric semilattice $S$ with respect to a contractive endomorphism $\phi: S \to S$. By analogy with the approach in \cite{DGBuatc}, we define the \emph{intrinsic semilattice entropy} $\hti:\Flow_{\Lqm}\to\R_{\geq0}\cup\{\infty\}$, where $\Lqm$ denotes the category of generalized quasimetric semilattices and their contractive homomorphisms. Moreover, for a category $\X$ and a functor $F:\Flow_\X\to \Flow_{\Lqm}$, we define the \emph{intrinsic functorial entropy} $\hti_F:\Flow_\X\to\R_{\geq0}\cup\{\infty\}$ by  $\hti_F=\hti\circ F$, and we show how the above mentioned specific intrinsic-like entropies can be obtained from this general scheme as intrinsic functorial entropies. Again, in almost all cases the functor $F:\Flow_\X\to\Flow_{\Lqm}$ is induced by a functor $\X\to\Lqm$.

\smallskip
The paper is organized as follows.
In Section~\ref{gqs-sec} we introduce the category $\Lqm$ we are mainly interested in, giving basic properties and examples.

In Section~\ref{dyn-sec} we start studying the dynamics of a generalized quasimetric semilattice $(S,d)\in \Lqm$. First, in \S\ref{f-inert}, we investigate the behavior of elements of $(S,d)$ under the action of a single contractive endomorphism $\phi$ and we define $\phi$-invariant and $\phi$-inert elements. Then, in \S\ref{fully-sec}, we introduce fully invariant, fully inert and uniformly fully inert elements of $(S,d)$ by analogy with \cite{BL,DDR,DDS}.
In \S\ref{traj-sec} we examine the properties of the trajectories of $\phi$-inert elements in order to introduce the intrinsic semilattice entropy in \S\ref{intrent-sec}.

Section~\ref{entropy-sec} is devoted to the study of the intrinsic semilattice entropy $\hti$. In \S\ref{basic-sec} we propose some basic properties of $\hti$ and we show that it is actually an invariant of the category $\Flow_{\Lqm}$ (see Corollary~\ref{inv:conj}). The whole \S\ref{LLsec} is dedicated to the so-called logarithmic law, that is, we try to answer the following question: given a contractive endomorphism $\phi:S\to S$ of a generalized quasimetric semilattice $S$ and $k\in\N$, is it true that $\widetilde h(\phi^k ) = k \cdot \widetilde h(\phi)$? 
%
%
We verify the inequality $\widetilde h(\phi^k ) \geq k \cdot \widetilde h(\phi)$ (see Corollary~\ref{ll1/2}), while the opposite one is proved only under some additional restraints. 
Trying to carry over to this framework the proof of the logarithmic law stated in \cite{DGSV} for the intrinsic algebraic entropy, an error was found in one of the steps
of the argument in \cite{DGSV}, and that proof has been corrected in \cite{TV}. Nevertheless, the argument used in \cite{TV} cannot be extended to our current setting.
We expect that the answer to the above general question is negative, but we did not find a counterexample yet.


In the final Section~\ref{known-sec} we put the general scheme to work and we show how the above mentioned specific intrinsic-like entropies can be recovered as intrinsic functorial entropies.

\subsection*{Dedication and acknowledgements}

This paper is dedicated to the memory of our friend and colleague Silvana Rinauro, whose contributions towards inertial properties in groups, obtained jointly with U. Dardano (see \cite{DDR,DR1,DR2,DR4}), triggered the key notion of $\phi$-inert subgroup, which is the core of the notion of intrinsic entropy.  

We warmly thank our friend and colleague Nicol\`o Zava and the referee for the useful comments and suggestions.

\subsection*{Notation and terminology}

We denote by $\Z$ the integers, by $\N$ the natural numbers and by $\N_+=\N\setminus\{0\}$ the positive integers. Moreover, $\R$ is the set of reals and $\R_{\geq0}=\{x\in \R\mid x\geq 0\}$.

Let $\mathfrak X$ be a category. With some abuse of notation we write $X\in\mathfrak X$ to say that $X\in\textup{Ob}(\mathfrak X)$. 
If $\mathfrak Y$ is a full subcategory of $\mathfrak X$, we briefly write $\mathfrak Y\subseteq \mathfrak X$.

A  \emph{flow}\index{flow} of $\mathfrak X$ is a pair $(X,\phi)$, where $X$ is an object of $\mathfrak X$ and $\phi: X\to X$ is an endomorphism in $\mathfrak X$.
A morphism between two flows $(X,\phi)$ and $(Y,\psi)$ of $\mathfrak X$ is a morphism $\alpha: X \to Y$ in $\mathfrak X$ such that $\psi\circ\alpha=\alpha\circ\phi$. 
This defines the category {$\Flow_{\mathfrak X}$} of flows of $\mathfrak X$. 

Clearly, in case $F:\mathfrak X\to \mathfrak Y$ is a functor, it induces a functor $\overline F:\Flow_\X\to \Flow_\mathfrak Y$ by letting $\overline F(X,\phi)=(F(X),F(\phi))$ for every $(X,\phi)\in\Flow_\X$ and $\overline F(\alpha)=F(\alpha)$ in case $\alpha:(X,\phi)\to(X',\phi')$ is a morphism in $\Flow_\X$.

\section{Generalized quasimetric semilattices and generalized normed semigroups}\label{gqs-sec}

\subsection{Semilattices with a generalized quasimetric}

Here we follow the approach from \cite{CDFGKT}.

\begin{definition}
A \emph{generalized quasimetric} on a non-empty set $S$ is a function $d: S \times S \to \R_{\geq0} \cup \{\infty\}$ such that:
\begin{enumerate}[(QM1)]
\item\label{qm1} for $x,y\in S$, $d(x,y)=d(y,x)=0$ if and only if $x=y$;
\item\label{qm2} for every $x,y,z\in S$, $d(x,z)\leq d(x,y)+d(y,z)$;
with the standard convention that $r<r + \infty = \infty + \infty = \infty$ for every $r\in\R_{\geq0}$.
\end{enumerate}
The pair $(S,d)$ is called \emph{generalized quasimetric space}.
\end{definition}

By analogy with the classical case of quasimetrics, we give the following natural definition.

\begin{definition}
Let $(S_1,d_1)$ and $(S_2,d_2)$ be generalized quasimetric spaces. Then a map $\alpha:S_1\to S_2$ is an \emph{isometry} if $d_2(\alpha(x),\alpha (y))=d_1(x,y)$ for every $x,y\in S_1$.
\end{definition}

For a generalized quasimetric space $(S,d)$, let $\leq_d$ be the partial order on $(S,d)$ defined by letting, for $x,y\in S$, 
\begin{equation}\label{0eq}
x\leq_d y\ \text{if and only if}\ d(y,x)=0.
\end{equation}
This is the dual of the specialization order of $d$.  

\begin{definition}[See \cite{CDFGKT}]\label{Def1} 
A generalized quasimetric space $(S,d)$ is a \emph{generalized quasimetric semilattice} if $(S,\leq_d)$ is a join-semilattice with bottom element $0$; the semilattice operation is denoted by $+$. 
Moreover, $(S,d)$ is \emph{invariant} if
\begin{enumerate}[(QM3)]
\item\label{qm3} $d(x,y)= d(x,x+y)$ for all $x,y\in S$.
\end{enumerate}
\end{definition}

From now on, whenever $(S,d)$ is assumed to be a  generalized quasimetric semilattice, we omit the appearance of the subscript in the notation, i.e.,  we use $\leq$ instead of $\leq_d$.

\smallskip
A semilattice with bottom element is a commutative monoid with all elements idempotent, so a semilattice homomorphism $\phi\colon S\to S'$ between two semilattices $S$ and $S'$ with bottom elements $0$ and $0'$ respectively, is a monoid homomorphism. Moreover, it is natural to define morphisms between invariant generalized quasimetric semilattices as follows.

\begin{definition}
A semilattice homomorphism $\phi:(S,d)\to (S',d')$ between two invariant generalized quasimetric semilattices is \emph{contractive} if $d'(\phi(x),\phi(y))\leq d(x,y)$ for every $x,y\in S$.
\end{definition}

Let $\Lqm$ denote the category of all invariant generalized quasimetric semilattices (i.e., satisfying~\ref{qm1},~\ref{qm2},~\ref{qm3}) and their contractive (semilattice) homomorphisms.

\medskip
If $(S,d)\in\Lqm$, then a simple application of~\ref{qm2} and~\ref{qm3} shows that  the function $d(-,y):S\to \R_{\geq0}$ is decreasing for every $y \in S$, while $d(x,-):S\to \R_{\geq0}$ is increasing for every $x \in S$,
that is:  
\begin{enumerate}[(M1)]
\item\label{1mon} if $x,x',y\in S$ and $x\leq x'$, then $d(x',y)\leq d(x,y)$;
\item\label{2mon} if $x,y,y'\in S$ and $y\leq y'$, then $d(x,y)\leq d(x,y')$.
\end{enumerate}

As further examples show, it is useful to allow the objects $(S,d)$ of $\Lqm$ to satisfy the  additional property: 
\begin{enumerate}[(OC)]
\item\label{x+y'} if $x,y,z\in S$ and $x\leq y\leq z$, then $d(x,z)=d(x,y)+d(y,z)$.
\end{enumerate}
In \cite{Sch}, $(S,d)\in\Lqm$ is called \emph{order-convex} if it satisfies \ref{x+y'}.
One can see that~\ref{x+y'} is equivalent to $d(x,y+y')=d(x,y)+d(x+y,y')$ for all triples $x,y,y'\in S$ (see \cite{CDFGKT}). 

As proved in \cite{CDFGKT}, \ref{qm3} is equivalent to
\begin{equation}\label{x+x'}
d(x+x',y+y')\leq d(x,y)+d(x',y')\ \text{for every}\ x,x',y,y'\in S,
\end{equation}
and an example is given witnessing that~\ref{qm3} is strictly weaker than~\ref{x+y'}.

Let $\barLqm$ be the full subcategory of $\Lqm$ with objects all $S\in\Lqm$ satisfying~\ref{x+y'}.


\subsection{The closeness relation}

\begin{definition}
Let $S\in \Lqm$. Two elements $x,y \in S$ are \emph{close}, denoted by $x \sim y $, if $d(x,y) < \infty$ and $d(y,x)<\infty$.
\end{definition}

It is easy to see that $\sim$ is an equivalence relation on $S\in\Lqm$ (the transitivity property holds by~\ref{qm2}).

\medskip
Let $(S,d)\in\Lqm$ and let
$$\F_d(S)=\{x\in S\mid d(0,x)<\infty\}\subseteq S.$$
Since by definition $d(x,0)=0$ for every $x\in S$, clearly $\F_d(S)=[0]_\sim$.

\begin{remark}
Let $S\in\Lqm$. Then $\sim$ is a congruence on $S$. In fact, for $x,x',y,y'\in S$, if $x' \sim x$ and $y' \sim y$, then also $x' + y' \sim x+y$ by \eqref{x+x'}.

Therefore, if $H$ is a subsemilattice of $S$, then so is $$H^{\sim} = \{x\in S\mid \exists y\in H,\ x\sim y\}=\bigcup_{y\in H}[y]_{\sim}.$$
In particular, if $S\in\Lqm$, then $\F_d(S) = \{0\}^{\sim}$ is a subsemilattice of $S$. 
\end{remark}

\subsection{Examples of generalized quasimetric semilattices}

Here we collect some examples that are used in Section~\ref{known-sec} (see also \cite{CDFGKT}).  

%

\begin{example}\label{ex:ab} 
Let $G$ be a group and denote by $\mathcal S(G)$ the family of all subgroups of $G$.  For $H,H' \in\mathcal S(G)$ with $H \subseteq H'$, the index of $H$ in $H'$ is denoted by $[H':H]$.
\begin{enumerate}[(a)] 
\item If $G$ is abelian, $\mathcal S(G)$ can be considered as a semilattice whose elements are partially ordered by inclusion and join-operation $H+H'$ for $H,H' \in\mathcal S(G)$. This gives a semilattice $\mathcal S^\vee(G)= (\mathcal S(G), +, \subseteq)$ with generalized quasimetric defined by 
$$d_{[\,\colon]}(H,H') = \log [H+H':H]\ \text{for} \ H,H' \in \mathcal S(G).$$

\item The set $\mathcal S(G)$ can be partially ordered by inverse inclusion even when $G$ is not necessarily abelian. Hence $\mathcal S^\wedge(G)=(\mathcal S(G),\cap,\supseteq)$ can be regarded as a semilattice with the operation $H \cap H'$ for $H,H'\in\mathcal S(G)$. In such a case, one has the generalized quasimetric defined by 
$$d^*_{[\,\colon]}(H,H') = \log [H:H\cap H']\ \text{for}\ H,H' \in \mathcal S(G).$$
\end{enumerate}
The generalized quasimetrics $d_{[\,\colon]}$ and $d^*_{[\,\colon]}$ satisfy all the properties~\ref{qm1},~\ref{qm2},~\ref{qm3} and~\ref{x+y'}; so $(\mathcal S^\vee(G),d_{[\,\colon]})\in\barLqm$ and $(\mathcal S^\wedge(G),d^*_{[\,\colon]})\in\barLqm$. 
Clearly, $d^*_{[\,\colon]}(H,H')=d_{[\,\colon]}(H',H)$ for all $H,H'\in \mathcal S(G)$ when $G$ is abelian, that is, $d_{[\,\colon]}$ coincides with the dual metric of $d^*_{[\,\colon]}$.

In both cases the closeness relation is known under the name \emph{commensurability}, that is, $H,H'\in\mathcal S(G)$ are commensurable if $[H:H\cap H']$ and $[H':H\cap H']$ are finite.  Moreover, $\F_{d_{[\,\colon]}}(\mathcal S^\vee(G))$ is the family of all finite subgroups and $\F_{d^*_{[\,\colon]}}(\mathcal S^\wedge(G))$ is the family of all finite-index subgroups of $G$.
\end{example}

The next obviously generalizes the previous example with $i(G) = \log |G|$.

\begin{example}\label{ex:inv}
Let $M$ be a unitary $R$-module, where $R$ is a unitary commutative ring. 
Now let $\mathcal S^\vee(M)$ be the lattice $\mathcal L(M)$ of all submodules of $M$, considered as a semilattice with operation $H + H'$ for $H,H' \in \mathcal L(M)$ and let $\mathcal S^\wedge(M)$ be the lattice $\mathcal L(M)$ considered as a semilattice with operation $H\cap H'$ for $H,H'\in\mathcal L(M)$. Fix a module invariant $i$, that is, $i(M)\in \R_{\geq0} \cup \{\infty\}$ and $i(M) = i(N)$ whenever $M \cong N$. Moreover, assume that $i$ is subadditive, that is, $i(M) \leq i(N) +  i(M/N)$ when $N$ is a submodule of $M$.  

Define the generalized quasimetrics $d_i$ on $\mathcal S^\vee(M)$ and $d_i^*$ on $\mathcal S^\wedge(M)$ by 
$$d_i(H,H') = i((H+H')/H)\ \text{and}\  d_i^*(H,H')=i(H/(H\cap H'))\ \text{for}\ H,H' \in S.$$

If $R$ is a field, then one is left with the only possible invariant $i = \dim_R$ and $M$ is a vector space over $R$.  Moreover, $d_i$ and $d_i^*$ satisfy all the properties~\ref{qm1},~\ref{qm2},~\ref{qm3},~\ref{x+y'}, and so $(\mathcal L(M),d_{\dim_R})\in\barLqm$ and $(\mathcal L(M),d_i^*)\in\barLqm$. Clearly, $\F_{d_{\dim_R}}(\mathcal S^\vee(M))$ is the family of all finite-dimensional subspaces of $M$ and $\F_{d^*_{\dim_R}}(\mathcal S^\wedge(M))$ is the family of all subspaces of $M$ with finite co-dimension.
\end{example}

\begin{remark}
In all cases considered above we have a concrete category $\X$ with a forgetful functor $U : \X \to {\bf Set}$ with plenty of nice properties.
For example, for $X\in\X$, the poset $\mathcal L(X)$ of all subobjects of $X$ in $\X$ is obtained from the lifting of subsets of $\mathcal P(U(X))$ along $U$.  
Hence, the meet in $\mathcal L(X)$ is simply the subobject with underlying set the intersection. 

In the above commutative examples $\mathcal L(X)$ is a complete lattice, so it has two semilattice structures which are related by an isomorphism or anti-isomorphism.
\end{remark}

\section{Dynamics in $\Lqm$}\label{dyn-sec}

\subsection{The $\phi$-invariant and $\phi$-inert elements}\label{f-inert}

In this section we study the interaction of single elements of some $(S,d)\in \Lqm$
with endomorphisms of $(S,d)$ in $\Lqm$. 


\begin{definition} Let $((S,d),\phi)\in\Flow_{\Lqm}$. An element $x\in S$ is called:
\begin{enumerate}[(i)]
\item \emph{$\phi$-invariant} if $d(x,\phi(x))=0$, (i.e., if $\phi(x)\leq x$);
\item \emph{$\phi$-inert} if $d(x,\phi(x))< \infty$. 
\end{enumerate}
We denote respectively by $\Inv_\phi(S)$ and $\mathcal I_\phi(S)$ the subsets of the $\phi$-invariant and the $\phi$-inert elements of $S$
(we shall see below that these are actually subsemilattices of $S$). Obviously, $\Inv_\phi(S) \subseteq \mathcal I_\phi(S)$.
\end{definition}

Next we see that 
large supply of $\phi$-inert elements is provided by the elements of $S$ close to $0$, shortly, 
$\F_d(S) \subseteq \mathcal I_\phi(S).$

\begin{remark}\label{F<I}
Let $((S,d),\phi)\in\Flow_{\Lqm}$ and $x\in S$. 
\begin{enumerate}[(a)]
\item If $x\in \F_d(S)$, then $\phi(x)\in \F_d(S)$ and $x\in \mathcal I_\phi(S)$. In fact, $d(0,x)<\infty$ implies $d(0,\phi(x))=d(\phi(0),\phi(x))\leq d(0,x)<\infty.$
Then $d(x,\phi(x))\leq d(0,\phi(x))$ by~\ref{1mon}.
\item The element $x$ is {$\phi$-invariant} precisely when $x=x+\phi(x)$. 
\end{enumerate}
\end{remark}

We show some properties of the $\phi$-inert elements, starting with the verification that $\mathcal I_\phi(S)$ is $\phi$-invariant, that is, $\phi(\mathcal I_\phi(S))\subseteq \mathcal I_\phi(S)$.

\begin{lemma}\label{subsemilat}
Let $((S,d),\phi)\in\Flow_{\Lqm}$. Then 
\begin{enumerate}[(a)]
\item $\phi^n(\mathcal I_\phi(S))\subseteq \mathcal I_\phi(S)$ for all $n\in\N$; in particular, if $x\in S$ is $\phi$-inert then $\phi(x)$ is $\phi$-inert; 
\item $\Inv_\phi(S)$ is a subsemilattice of $S$;
\item 
$\mathcal I_\phi(S)$ is a subsemilattice of $S$.
\end{enumerate}
\end{lemma}
\begin{proof} 
(a) It suffices to observe that $d(\phi^n(x),\phi^{n+1}(x))\leq d(x,\phi(x))$.

(b)  If $x,y	\in S$ are $\phi$-invariant, then $x+y$ is $\phi$-invariant by Remark~\ref{F<I}(b).

(c) By \eqref{x+x'}, one has $0\leq d(x+y,\phi(x+y))=d(x+y,\phi(x)+\phi(y))\leq d(x,\phi(x))+d(y,\phi(y))$ for every $x,y\in S$. Therefore, if $x,y\in S$ are $\phi$-inert, then $x+y$ is $\phi$-inert.
\end{proof}

\subsection{Fully invariant, fully inert and uniformly fully inert elements}\label{fully-sec}

Inspired by the notions introduced and studied in \cite{DDR,DDS}, we give the following.

\begin{definition} 
Let $(S,d)\in\Lqm$. An element $x\in S$ is called:
\begin{enumerate}[(i)]
\item \emph{fully invariant} if $x$ is $\phi$-invariant for every contractive endomorphism $\phi$ of $S$;
\item \emph{fully inert} if $x$ is $\phi$-inert for every contractive endomorphism $\phi$ of $S$;
\item \emph{uniformly fully inert} if there exists $C>0$ such that $d(x, \phi(x))\leq C$ for every contractive endomorphism $\phi$ of $S$.
\end{enumerate}
\end{definition}

In the sequel, given $(S,d)\in\Lqm$, we denote by:
\begin{enumerate}[(i)]
\item $\mathcal I (S)$  the set of all fully inert elements of $S$; 
\item $\Inv(S)$ the set of all fully invariant elements of $S$;
\item $\mathcal I_u(S)$ the set of all uniformly fully inert elements of $S$.
\end{enumerate}
Clearly, $\Inv(S)  \subseteq \mathcal I_u(S)  \subseteq \mathcal I(S)$ and 
$$\mathcal I(S) =\bigcap_{\phi\in\mathrm{End}(S)}\mathcal I_\phi(S), \quad \Inv(S)=\bigcap_{\phi\in\mathrm{End}(S)}\Inv_\phi(S).$$
By Lemma~\ref{subsemilat}(b), $\Inv(S)$ is a subsemilattice of $S$, $\mathcal I(S)$ is a subsemilattice of $S$ by Lemma~\ref{subsemilat}(c), and a similar argument shows that also $\mathcal I_u(S)$ is a subsemilattice of $S$.
 
\begin{lemma}\label{Iu}
Let $(S,d)\in\Lqm$, and $x,y\in S$ with $x\sim y$. 
\begin{enumerate}[(a)]
\item If $\phi$ is a contractive endomorphism of $S$ and $x \in \mathcal I_\phi(S)$, then $y \in \mathcal I_\phi(S)$.
\item If $x \in \mathcal I_u(S)$, then $y \in \mathcal I_u(S)$ (in particular, if $x \in \Inv(S)$, then $y \in \mathcal I_u(S)$).
\end{enumerate}
\end{lemma}
\begin{proof}
(a) By~\ref{qm2} and the fact that $\phi$ is contractive, 
\[d(y,\phi(y))\leq d(y,x)+d(x,\phi(x))+d(\phi(x),\phi(y))\leq d(y,x)+d(x,\phi(x))+d(x,y)<\infty.\]

(b) Similarly, if for some $C$ one has $d(x,\phi(x))\leq C$ for all contractive endomorphisms $\phi$ of $S$, then using the fact that $\phi$ is contractive and~\ref{qm2}, we have 
\begin{align*}d(y,\phi(y))&\leq d(y,x)+d(x,\phi(x))+d(\phi(x),\phi(y))\leq  \\ &\leq d(y,x)+d(x,\phi(x))+d(x,y) \leq d(y,x)+d(x,y) + C.\qedhere\end{align*}
\end{proof}

Let us consider the set $\Inv(S)^\sim$ of all elements $x\in S$ which are close to some $y \in \Inv(S)$.  
In Lemma~\ref{Iu}(b) we showed that $\Inv(S)^\sim \subseteq \mathcal I_u(S)$. 
It is not clear whether one can invert this inclusion, namely: 

\begin{question}\label{ques1}
Let $(S,d)\in\Lqm$. If $y \in \mathcal I_u(S)$, does there exist $x \in \Inv(S)$ such that $x \sim y$? In other words, does the equality $\Inv(S)^\sim=\mathcal I_u(S)$ hold?
\end{question}

\begin{remark}
The above notions come from the case of groups, that is, in case $G$ is a group, one considers the lattice $\mathcal S(G)$ of all its subgroups with the generalized quasimetric discussed in Example~\ref{ex:ab}(b).

(a) Since fully invariant subgroups are usually hard to come by, one relaxes the property of fully invariance defining a subgroup $H$ of $G$ to be \emph{characteristic} in $G$ if $H$ is $\phi$-invariant for every $\phi\in\Aut(G)$. 

In this case one may also choose some other subgroup of $\Aut(G)$; in particular, if this subgroup of $\Aut(G)$ is $\mathrm{Inn}(G)$, then the subgroups $H$ of $G$ with $\phi(H)\leq H$ for every $\phi\in\mathrm{Inn}(G)$ are obviously the normal ones.

(b) More in general, for an \emph{operator group} or \emph{$\Omega$-group} $G$, that is, a group $G$ equipped with a family $\Omega$ of endomorphisms of $G$,  a subgroup $H$ of $G$ is called \emph{$\Omega$-invariant} or \emph{$\Omega$-admissible} if $\phi(H)\subseteq H$ for every $\phi\in\Omega$.  In particular, with $\Omega= \mathrm{Inn}(G)$ (respectively, $\Omega=\mathrm{End}(G)$, $\Omega=\mathrm{Aut}(G)$), the $\Omega$-admissible subgroups of $G$ are precisely the normal (respectively, the fully invariant, the characteristic) ones. 

(c) Analogously to the discussion in item (a), the subgroups $H$ of $G$ that are ``fully inert with respect to $\mathrm{Inn}(G)$'', that is, those $H$ that are $\phi$-inert for every $\phi\in\mathrm{Inn}(G)$ where studied under the name \emph{inert subgroups} in the nineties. Clearly, fully inert subgroups in the above sense are inert. This triggered the introduction of fully inert subgroups of abelian groups in \cite{DGSV1}.

(d) Bergman and Lenstra \cite{BL} introduced the notion of \emph{uniformly inert} subgroups of $G$. These are the subgroups $H$ of $G$ that are ``uniformly inert with respect to $\mathrm{Inn}(G)$'', that is, those $H$ such that for some constant $C>0$, $[H:\phi(H)\cap H]\leq C$ for every $\phi\in\mathrm{Inn}(G)$. Clearly, every uniformly fully inert subgroup is uniformly inert.

It is known from \cite[Theorem 3]{BL}  that a subgroup of a group $G$  is uniformly inert if and only if it is commensurable with a normal subgroup of $G$.  Nevertheless, Question~\ref{ques1} is still open; it was raised in \cite{DDR,DDS} in the case of a group $G$ and semilattice $S=\mathcal S(G)$ with the generalized quasimetric described in Example~\ref{ex:ab}(b). 
\end{remark}

\subsection{Trajectories and their properties}\label{traj-sec}

In this subsection we investigate the properties of the $\phi$-trajectories of $\phi$-inert elements, which turn out to be $\phi$-inert elements (see Lemma~\ref{wannabelogarithmic law}).

\begin{definition}
Let $((S,d),\phi)\in\Flow_{\Lqm}$ and $x\in S$. For $n\in\N_+$, the \emph{$n$-th $\phi$-trajectory} of $x$ is 
$T_n(\phi,x)=x+\phi(x)+\ldots+\phi^{n-1}(x)\in S$ and let $T_0(\phi,x)=0$.
\end{definition}

We simply write $T_n$ in place of $T_n(\phi,x)$, when $\phi$ and $x$ are clear from the context. 

\begin{remark}\label{rem:Tn+m}
Let $((S,d),\phi)\in\Flow_{\Lqm}$ and $x\in S$. For every $n,m,i\in\N_+$, $i\leq m$, it is straightforward to see that
$T_n(\phi^i,T_{m}(\phi,x))=T_{(n-1)i+m}(\phi,x).$
\end{remark}

The implication (a)$\Rightarrow$(c) in the next result is in Lemma~\ref{subsemilat}(a). 

\begin{proposition}\label{inert:corollary}
Let $((S,d),\phi)\in\Flow_{\Lqm}$ and $x\in S$. Then the following conditions are equivalent:
\begin{itemize}
\item [(a)] $x$ is $\phi$-inert (i.e., $d(x,\phi(x))< \infty$);
\item [(b)] $d(x,T_n(\phi,x))< \infty$ for every $n\in \N_+$;
\item [(c)] $x$ is $\phi^n$-inert (i.e., $d(x,\phi^n(x))< \infty$) for every $n\in \N_+$.
\end{itemize}
\end{proposition}
\begin{proof}
(c)$\Rightarrow$(a) This is trivial.

(b)$\Rightarrow$(c) For $n\in\N_+$, \ref{2mon} gives
\(d(x,\phi^{n-1}(x)) \leq d(x,T_n)\).

(a)$\Rightarrow$(b) For $n\in\N_+$,
$$d(x,T_n)\leq d(x,\phi(x))+d(\phi(x),\phi(T_{n-1}))\leq d(x,\phi(x))+d(x,T_{n-1}).$$
By induction on gets $d(x,T_n)\leq (n-1)d(x,T_{n-1})$.
Hence, by \ref{2mon}, \(d(x,\phi^{n-1}(x))\leq d(x,T_n)\leq (n-1)d(x,\phi(x))<\infty\). 
\end{proof}

The above proposition implies in particular that $\mathcal I_\phi(S) = \bigcap_{n\in \N_+}\mathcal I_{\phi^n}(S).$

\begin{lemma}\label{alphan}
Let $((S,d),\phi)\in\Flow_{\Lqm}$ and let $x\in S$ be $\phi$-inert. Then, for every $n\in\N_+$,
$d(T_n(\phi,x),T_{n+1}(\phi,x))\leq d(T_{n-1}(\phi,x),T_n(\phi,x)).$
So the sequence $\{d(T_n,T_{n+1})\}_{n \in \N}$ of non-negative reals is decreasing.
\end{lemma}
\begin{proof}
Fix $n\in\N_+$. Since $\phi(T_{n-1})\leq x + \phi(T_{n-1}) = T_n$,~\ref{qm3} gives
\begin{align*}d(T_n,T_{n+1})&=d(x+\phi(T_{n-1}),x+\phi(T_{n-1})+\phi(T_n))= \\&=d(x+\phi(T_{n-1}),\phi(T_n)) \leq d(\phi(T_{n-1}),\phi(T_n))\leq d(T_{n-1},T_n).\qedhere\end{align*}
\end{proof}

The next result is useful in \S\ref{LLsec} about the so-called logarithmic law.

\begin{lemma}\label{wannabelogarithmic law}
Let $((S,d),\phi)\in\Flow_{\Lqm}$.
If $x$ is $\phi^k$-inert for some $k \in \N$, then $T_k(\phi,x)$ is $\phi$-inert and so $\phi^k$-inert. 

In particular, if $x$ is $\phi$-inert, then $T_n(\phi,x)$ is $\phi$-inert for all $n\in\N$.
\end{lemma}
\begin{proof}
Let $k\in\N$ and $x$ be $\phi^k$-inert. Since $$T_k(\phi,x) + \phi(T_k(\phi,x)) = T_{k+1}(\phi,x) = T_k(\phi,x) + \phi^k(x),$$~\ref{qm3} implies that $$d ( T_k(\phi,x), \phi(T_k(\phi,x)) ) = d ( T_k(\phi,x), T_k(\phi,x) + \phi^k(x)) = d ( T_k(\phi,x), \phi^k(x)).$$ Then $d ( T_k(\phi,x), \phi(T_k(\phi,x)) ) \leq d ( x, \phi^k(x))$ by~\ref{1mon}, so $T_k(\phi,x)$ is $\phi$-inert. 

The remaining part is a consequence of Proposition~\ref{inert:corollary}.
%
%
\end{proof}


\begin{remark}
Let $((S,d),\phi)\in\Flow_{\Lqm}$. If $x,y\in S$ are $\phi$-inert and $n\in\N$,  then $d(T_n(\phi,x),T_n(\phi,y))\leq n d(x,y)$ by \eqref{x+x'}.
Consequently, for every $m\in\N$, since $T_{n+m}(\phi,x)=T_n(\phi,T_{m+1}(\phi,x))$ by Remark~\ref{rem:Tn+m}, $$d(T_n(\phi,x),T_{n+m}(\phi,x))\leq nd(x,T_{m+1}(\phi,x)).$$
\end{remark}

\begin{lemma}\label{TnTn+1}
Let $((S,d),\phi)\in\Flow_{\barLqm}$ and $x\in S$. Then, for every $n,m\in\N$, $d(x,T_{n+m})=d(x,T_n)+d(T_n,T_{n+m}).$
\end{lemma}
\begin{proof} 
Since $x\leq T_n\leq T_{n+m}$, the assertion follows from~\ref{x+y'}.
\end{proof}

\subsection{The intrinsic semilattice entropy}\label{intrent-sec}

The next result enables us to introduce the fundamental notion of this paper, namely the intrinsic semilattice entropy of a contractive endomorphism $\phi$ of  an object $(S,d)$ of $\Lqm$.

\begin{theorem}\label{existence:limit}
Let $((S,d),\phi)\in\Flow_{\Lqm}$. The following limit exists for every $x\in \mathcal I_\phi(S)$:
$$\widetilde h(\phi,x) = \lim_{n\to\infty} \frac{d(x,T_n(\phi,x))}{n}.$$
\end{theorem}

This important result is a consequence of the following proposition and Fekete Lemma (see \cite{Fek}).

\begin{proposition}\label{prop:subadd}
Let $((S,d),\phi)\in\Flow_{\Lqm}$ and $x\in \mathcal I_\phi(S)$. 
Then  $\{d(x,T_{n+1}(\phi,x))\}_{n\in\N}$ is subadditive.
\end{proposition}

\begin{proof} For $n \in \N$ let $c_n= d(x,T_{n+1}(\phi,x))$. We have to prove that $c_{m+n}\leq c_m + c_n$ for every $m,n\in\N$. One has $$c_{m+n}=d(x,T_{m+n+1}(\phi,x))\leq c_n+d(T_{n+1}(\phi,x),T_{m+n+1}(\phi,x))$$ by~\ref{qm2}. Hence, to conclude that $c_{m+n}\leq c_m + c_n$, it suffices to compute
\begin{align*}
 d(T_{n+1}(\phi,x),T_{m+n+1}(\phi,x))&=d(T_{n+1}(\phi,x),T_{n+1}(\phi,x)+\phi^{n+1}(T_m(\phi,x))) \\
                                               &=d(T_{n+1}(\phi,x),\phi^{n+1}(T_m(\phi,x))) \\
                                               &\leq d(\phi^n(x),\phi^{n+1}(T_{m}(\phi,x))) \\
                                                &\leq d(x,\phi (T_{m}(\phi,x)))\\
                                                &\leq d(x,T_{m+1}(\phi,x)) = c_m,                                               
\end{align*}
where the first equality holds by definition, the second by~\ref{qm3}, the first inequality by~\ref{1mon} since $\phi^n(x)\leq  T_{n+1}(\phi,x)$, the second inequality because $\phi$ is contractive, and the last inequality by~\ref{2mon} since $\phi(T_m(\phi,x))\leq  T_{m+1}(\phi,x)$.
\end{proof}

Theorem~\ref{existence:limit} allows us to give the main definition of this paper.

\begin{definition}
Let $((S,d),\phi)\in\Flow_{\Lqm}$.  The \emph{intrinsic semilattice entropy} of $\phi$ with respect to $x\in \mathcal I_\phi(S)$ is the value $\widetilde h(\phi,x)$ introduced in Theorem~\ref{existence:limit}. 

The \emph{intrinsic semilattice entropy} of $\phi$ is $\widetilde h(\phi) = \sup\{\widetilde h(\phi,x)\mid x\in\mathcal I_\phi(S)\}$.
\end{definition}

Due to Lemma~\ref{TnTn+1} we see now that a stronger result with respect to Theorem~\ref{existence:limit} holds for flows in $\barLqm$. 

\begin{proposition}
Let $((S,d),\phi)\in\Flow_{\barLqm}$ and the value set $d(S\times S)$ be a well-ordered subset of the range. 
If $x\in S$ is $\phi$-inert, then 
$$\widetilde h(\phi,x) =\inf_{n\in \N}d(T_n(\phi,x),T_{n+1}(\phi,x))\in\R_{\geq0}.$$
\end{proposition}
\begin{proof}
By Lemma~\ref{alphan} the sequence $\{d(T_n,T_{n+1})\}_{n\in\N}$ is decreasing, so it stabilizes, according to our hypothesis. Let  $\alpha=\inf\{d(T_n,T_{n+1})\mid n\in \N\}$.
There exists $n_0\in\N$  such that $d(T_n,T_{n+1})=\alpha$ for every $n\in\N$ with $n\geq n_0$.
By Lemma~\ref{TnTn+1}, $d(x,T_{n_0+m})=d(x,T_{n_0})+m\alpha$ for every $m\in\N$; therefore, 
\[\widetilde h(\phi,x)= \lim_{m\to\infty} \frac{d(x,T_{n_0+m}(\phi,x))}{n_0+m}= \lim_{m\to\infty} \frac{d(x,T_{n_0})+m\alpha}{n_0+m}=\alpha.\qedhere\]
\end{proof}

\section{Basic properties of the intrinsic semilattice entropy}\label{entropy-sec}

In this section we investigate several properties of the map $\hti\colon \Flow_{\Lqm}\to \R_{\geq0}\cup\{\infty\}$, where we let $\hti(\phi)=\hti(S,\phi)$ for every $(S,\phi)\in\Flow_{\Lqm}$.

\subsection{The intrinsic semilattice entropy is an invariant}\label{basic-sec}

We start by showing that the identity map has zero intrinsic semilattice entropy.

\begin{example}
If $S\in{\Lqm}$, then $\widetilde h(id_S)=0$.
Indeed, every $x\in S$ is $id_S$-inert, and $T_n (id_S,x) = x$ for every $n\in\N$, so $\widetilde h(id_S,x)=0$.
\end{example}


The condition needed in item~\ref{ei1} of the next result seems to be different from the surjectivity of $\alpha:S_1 \to S_2$.

\begin{proposition}\label{entropies:inequalities}
Let $\alpha:((S_1,d_1),\phi_1) \to ((S_2,d_2),\phi_2)$ be a morphism in $\Flow_{\Lqm}$.
Then $\alpha ( \mathcal I_{\phi_1}(S_1) ) \subseteq \mathcal I_{\phi_2}(S_2)$ and $T_n(\phi_2,\alpha (x)) = \alpha ( T_n(\phi_1, x) )$ for every $x\in S_1$ and  $n \in \N$. Moreover:
\begin{enumerate}[(a)]
\item\label{ei1} if $\alpha ( \mathcal I_{\phi_1}(S_1) ) = \mathcal I_{\phi_2}(S_2)$, then $\widetilde h(\phi_2) \leq \widetilde h(\phi_1)$;
\item\label{ei2} if $\alpha$ is an injective isometry, then $\widetilde h(\phi_2) \geq \widetilde h(\phi_1)$.
\end{enumerate}
\end{proposition}
\begin{proof}
Since $\alpha$ is  a contractive semilattice homomorphism such that $ \alpha \phi_1 = \phi_2 \alpha$, one has
$d_2 ( \alpha(x), \phi_2 (\alpha(x))) =  d_2 ( \alpha(x), \alpha (\phi_1(x))) \leq d_1(x, \phi_1(x) )<\infty.$ Then  $\alpha(x)\in S_2$ is $\phi_2$-inert whenever $x \in S_1$ is $\phi_1$-inert. 
 
If $x\in S_1$ and  $n \in \N$, then 
\begin{align*}
T_n(\phi_2,\alpha (x)) &= \alpha (x) + \phi_2 \alpha (x) +\cdots +\phi_2^{n-1 }\alpha (x)= \\ &= \alpha (x) + \alpha  \phi_1 (x) +\cdots +\alpha  \phi_1^{n-1} (x)
=\alpha ( T_n(\phi_1, x) ).
\end{align*}
(a) Let $y\in \mathcal I_{\phi_2}(S_2)$, and let $x\in \mathcal I_{\phi_1}(S_1)$ be such that $y = \alpha (x)$.
Using the first part of the proof, we obtain
$$\widetilde h(\phi_2, y ) =\lim_{n\to\infty} \frac{d_2(\alpha(x),\alpha ( T_n(\phi_1, x) ))}{n} \leq \lim_{n\to\infty} \frac{d_1(x, T_n(\phi_1, x) )}{n} = \widetilde h(\phi_1, x ).$$
Since $\widetilde h(\phi_1, x )\leq \widetilde h(\phi_1)$, taking the supremum over $y\in \mathcal I_{\phi_2}(S_2)$ in the above inequality we get $\widetilde h(\phi_2) \leq \widetilde h(\phi_1)$.

(b) Assume that $\alpha$ is injective and $d_2 (\alpha(x), \alpha (y)) = d_1(x, y )$ for every $x,y \in S_1$. For a $\phi_1$-inert element $x \in S_1$, we proved already that $\alpha(x)\in S_2$ is $\phi_2$-inert. Moreover, 
$$
\widetilde h(\phi_2, \alpha(x) ) =\lim_{n\to\infty} \frac{d_2(\alpha(x),\alpha ( T_n(\phi_1, x) ))}{n} = 
\lim_{n\to\infty} \frac{d_1(x, T_n(\phi_1, x) )}{n} = \widetilde h(\phi_1, x ).
$$
Then $\widetilde h(\phi_2) \geq \widetilde h(\phi_2, \alpha(x) ) = \widetilde h(\phi_1, x )$ for every $\phi_1$-inert element $x$, so $\widetilde h(\phi_2) \geq  \widetilde h(\phi_1)$.\qedhere 
\end{proof}

When $\alpha:((S_1,d_1),\phi_1) \to ((S_2,d_2),\phi_2)$ is an isomorphism in $\Flow_{\Lqm}$, it satisfies all the hypotheses in Proposition~\ref{entropies:inequalities}(a,b). Moreover, $\phi_2$ coincides with $\alpha \phi_1 \alpha^{-1}$, so $\widetilde h( \alpha \phi_1 \alpha^{-1} ) = \widetilde h(\phi_1)$ in this case.

\begin{corollary}[Invariance under conjugation]\label{inv:conj}
Let $\alpha:((S_1,d_1),\phi_1) \to ((S_2,d_2),\phi_2)$ be an isomorphism in $\Flow_{\Lqm}$.
Then $\alpha ( \mathcal I_{\phi_1}(S_1) ) = \mathcal I_{\phi_2}(S_2)$ and $\widetilde h(\phi_2) = \widetilde h(\phi_1)$.
\end{corollary}

This shows that $\hti\colon \Flow_{\Lqm}\to \R_{\geq0}\cup\{\infty\}$ is an invariant of $\Flow_{\Lqm}$.

\subsection{Towards the logarithmic law}\label{LLsec}

In the following results we compare the intrinsic semilattice entropy $\widetilde h(\phi)$ of a flow $((S,d),\phi)$ in $\Flow_{\Lqm}$, with the intrinsic  semilattice  entropy of the composition flow $((S,d),\phi^k)$.

\begin{lemma}\label{1/2logarithmic law}
Let $((S,d),\phi)\in\Flow_{\Lqm}$ and $k \in \N$. If $x$ is ${\phi}$-inert, then 
\(\widetilde h(\phi, T_k(\phi,x) )= \widetilde h(\phi, x).\)
\end{lemma}
\begin{proof}
Let $k\in \N$ and $x\in S$ be ${\phi}$-inert. Then $T_k(\phi,x)$ is $\phi$-inert by Lemma~\ref{wannabelogarithmic law}. 
%
%
Let $n \in \N_+$. Remark~\ref{rem:Tn+m} gives
\begin{align*}\widetilde h(\phi, T_k(\phi,x) ) &=\lim_{n\to\infty} \frac{d( T_k(\phi,x), T_{n}(\phi,T_k(\phi,x)))}{n} 
\\&= \lim_{n\to\infty} \frac{d( T_k(\phi,x) ,T_{n+k-1}(\phi,x))}{n}.\end{align*}
Since $x\leq T_k(\phi,x)$,  $d( T_k(\phi,x), x)=0$ by \eqref{0eq}. Then by~\ref{qm2}, we obtain
\begin{align*}
\widetilde h(\phi, T_k(\phi,x) )& \leq
\lim_{n\to\infty} 
\frac{d( T_k(\phi,x), x)}{n}
 + 
\lim_{n\to\infty} \frac{d( x ,T_{n+k-1}(\phi,x))}{n} = \\
&= \lim_{n\to\infty} \frac{d( x ,T_{n+k-1}(\phi,x))}{n+k-1}\cdot\frac{n+k-1}{n}= \widetilde h(\phi,x).
\end{align*}
On the other hand,
\begin{align*}\widetilde h(\phi,x) &= \lim_{n\to\infty} \frac{d( x, T_{n+k}(\phi,x))}{n+k} \\& \leq\lim_{n\to\infty} \frac{d( x, T_{k}(\phi,x))}{n+k} +\lim_{n\to\infty} \frac{d( T_{k}(\phi,x), T_{n+k}(\phi,x))}{n+k}.\end{align*}
As $d( x, T_{k}(\phi,x))\in \R$ and does not depend on $n$, Remark~\ref{rem:Tn+m} gives
\[\widetilde h(\phi,x) \leq \lim_{n\to\infty} \frac{d( T_{k}(\phi,x), T_{n+1}(\phi, T_{k}(\phi,x) ))}{n+1}\frac{n+1}{n+k} = \widetilde h(\phi,T_{k}(\phi,x)).\qedhere\]
\end{proof}

Then we obtain some sort of ``local'' logarithmic law passing to the trajectories.

\begin{proposition}\label{lemma0}
Let $((S,d),\phi)\in\Flow_{\Lqm}$ and $k \in \N$. 
If $x$ is $\phi^k$-inert, then
\begin{equation}\label{traj:eq}
\widetilde h(\phi^k, T_k(\phi,x) ) = k \cdot \widetilde h(\phi, T_k(\phi,x) ).
\end{equation}
Moreover, if $x$ is $\phi$-inert, then
\begin{equation}\label{traj:NEQ}
\widetilde h(\phi^k, T_k(\phi,x) ) = k \cdot \widetilde h(\phi, T_k(\phi,x) ) =  k \cdot \widetilde h(\phi, x ).
\end{equation}
\end{proposition}
\begin{proof}
First assume that $x$ is $\phi^k$-inert. Then $T_k(\phi,x) $ is $\phi$-inert and $\phi^k$-inert by Lemma~\ref{wannabelogarithmic law}. Let $n \in \N_+$. By Remark~\ref{rem:Tn+m},
\begin{equation}\label{traj:eq1}
T_{nk}(\phi,x)=T_{n}(\phi^k, T_k(\phi,x)) = T_{kn-k+1}(\phi,T_k(\phi,x)).
\end{equation}
Then we get \eqref{traj:eq} as
\begin{align*}
\widetilde h(\phi^k, T_k(\phi,x) ) &=
\lim_{n\to\infty} \frac{d( T_k(\phi,x), T_{n}(\phi^k,T_k(\phi,x)))}{n}\\
&=\lim_{n\to\infty} \frac{d( T_k(\phi,x) ,T_{kn-k+1}(\phi,T_k(\phi,x)) )}{kn-k+1}\frac{kn-k+1}{n}\\ &=k \cdot \widetilde h(\phi, T_k(\phi,x) ).
\end{align*}

Now assume that $x$ is $\phi$-inert. Then $x$ is $\phi^k$-inert as well, so \eqref{traj:eq} ensures the first equality in \eqref{traj:NEQ}.
Moreover, Lemma~\ref{1/2logarithmic law} applies to provide the second equality in \eqref{traj:NEQ}.
\end{proof}

As an immediate consequence of Proposition~\ref{lemma0} we obtain: 

\begin{corollary}\label{ll1/2}
If $((S,d),\phi)\in\Flow_{\Lqm}$ and $k \in \N$, then $k \cdot \widetilde h(\phi) \leq \widetilde h(\phi^k).$
\end{corollary}
\begin{proof}
Let $x\in \mathcal I_\phi(S)$. 
By \eqref{traj:NEQ}, $k\cdot\widetilde h(\phi, x )  = \widetilde h(\phi^k, T_k(\phi,x) ) \leq \widetilde h(\phi^k)$. 
Thus, $k \cdot \widetilde h(\phi) \leq \widetilde h(\phi^k)$ by taking the supremum over all $x\in\mathcal I_\phi(S)$.
\end{proof}

In the rest of this subsection we give partial results concerning the converse inequality $\widetilde h(\phi^k ) \leq k \cdot \widetilde h(\phi)$. 
We start from a ``local'' version generalizing Proposition~\ref{lemma0}, where we replace the $\phi$-inert element $T_k(\phi,x)$ that appears in \eqref{traj:eq} with a generic $\phi$-inert element of $S$.

\begin{lemma}\label{uselesslemma} 
If $((S,d),\phi)\in\Flow_{\Lqm}$, $k \in \N$ and $x$ is $\phi$-inert, then 
$\widetilde h(\phi^k, x ) \leq k \cdot \widetilde h(\phi,x)$.
\end{lemma}
\begin{proof}
Note first that (even in case $x$ is not $\phi$-inert), $T_n(\phi^k, x) \leq T_{kn-k+1}(\phi,x)$. Then
\begin{equation*}	
\widetilde h(\phi^k, x ) 
\leq \lim_{n\to\infty} \frac{d(x, T_{kn-k+1}(\phi,x))}{kn-k+1}\frac{kn-k+1}{n} 
= k \cdot \widetilde h(\phi,x).\qedhere
\end{equation*}
\end{proof}

The next corollary gives a precise description of $k\cdot \hti(\phi)$ and covers, in particular, Corollary~\ref{ll1/2}.

\begin{corollary}\label{uselesscorollary}
Let $((S,d),\phi)\in\Flow_{\Lqm}$, and $k \in \N$. 
Then 
\[ k\cdot \hti(\phi)=\sup\{\widetilde h(\phi^k,x)\mid x\in\mathcal I_\phi(S)\}\leq \hti(\phi^k).\]
\end{corollary}
\begin{proof} 
The second inequality follows from the fact that $\mathcal I_\phi(S)\subseteq \mathcal I_{\phi^k}(S)$.
Let $x\in\mathcal I_\phi(S)$; then $y = T_k(\phi,x)\in\mathcal I_\phi(S)$ by Proposition~\ref{inert:corollary}. Respectively from Lemma~\ref{uselesslemma} and Lemma~\ref{1/2logarithmic law}, it follows that
$\widetilde h(\phi^k, x ) \leq k \cdot \widetilde h(\phi,x ) = k \cdot \widetilde h(\phi, y )$.  So, $\sup\{\widetilde h(\phi^k,x)\mid x\in\mathcal I_\phi(S)\}\leq k\cdot \hti(\phi)$. 
To prove the converse inequality, apply \eqref{traj:NEQ} to obtain 
$$k \cdot \widetilde h(\phi,x)=\widetilde h(\phi^k,T_k(\phi,)) \leq \sup\{\widetilde h(\phi^k,x)\mid x\in\mathcal I_\phi(S)\}.$$
Hence, $k\cdot \hti(\phi)\leq \sup\{\widetilde h(\phi^k,x)\mid x\in\mathcal I_\phi(S)\}$.
\end{proof}

Corollary~\ref{uselesscorollary} implies that the logarithmic law holds in the following special cases. 

\begin{corollary}
Let $((S,d),\phi)\in\Flow_{\Lqm}$, and $k \in \N$. If either $\hti(\phi^k) =0$ or $\mathcal I_\phi(S)=\mathcal I_{\phi^k}(S)$, then $\hti(\phi^k) = k \cdot \widetilde h(\phi).$
\end{corollary}

As $\mathcal I_\phi(S) \subseteq \mathcal I_{\phi^k}(S)$ holds in general by Proposition~\ref{inert:corollary}, $\mathcal I_\phi(S)=\mathcal I_{\phi^k}(S)$ occurs for example when $\mathcal I_\phi(S)=S$. This is the case when the generalized quasimetric $d$ is a quasimetric (that is, $d$ takes only finite values), and so we obtain the following instance of the logarithmic law.

\begin{corollary} 
Let $((S,d),\phi)\in\Flow_{\Lqm}$ with $d$ a quasimetric, and let $k \in \N$. Then $\hti(\phi^k) = k \cdot \widetilde h(\phi)$.
\end{corollary}

\section{Obtaining the specific entropy functions}\label{known-sec}

In the next subsections of this section we use the following scheme in order to find the known intrinsic-like entropies as intrinsic functorial entropies.

\subsection{Intrinsic functorial entropy}

As recalled in the introduction, for $\X$ a category and $F:\Flow_\X \to \Flow_{\Lqm}$ a functor, the \emph{intrinsic functorial entropy} $\hti_F$ associated to $F$ is defined by letting $\hti_F=\hti\circ F$.
We set $\hti_F(\phi)=\hti_F(X,\phi)$ for every $(X,\phi)\in\Flow_\X$ as usual.

\smallskip
The following shows that $\hti_F$ is an invariant of $\Flow_\X$.

\begin{proposition}
For every functor $F:\Flow_\X\to \Flow_{\Lqm}$, the intrinsic functorial entropy $\hti_F$ is invariant under conjugation, that is, for every $(X,\phi),(Y,\psi)\in\Flow_\X$ such that there exists an isomorphism $\alpha:(X,\phi)\to(Y,\psi)$ one has $\hti_F(\phi)=\hti_F(\psi)$.
\end{proposition}
\begin{proof} Assume that $F:\Flow_\X\to \Flow_{\Lqm}$ is covariant.
By hypothesis, $\psi= \alpha\circ \phi \circ \alpha^{-1}$. Then $F(\psi) = F(\alpha)\circ F(\phi) \circ F(\alpha)^{-1}$ in $\Lqm$. By Corollary~\ref{inv:conj}, $\hti_F(\psi)=\hti(F(\psi)) = \hti(F(\phi))= \hti_F(\phi)$. For a contravariant functor $F$ one can proceed analogously.
\end{proof}

\subsection{Intrinsic (adjoint) algebraic entropy}

Let $G$ be an abelian group and let $f\colon G \to G$ be an endomorphism. A subgroup $H$ of $G$ is \emph{$f$-inert} if $\card{(H + f(H))/H}$ is finite. The  family $\mathcal I_f(G)$ of the $f$-inert subgroups of $G$ contains all finite subgroups, all finite-index subgroups, as
well as all $f$-invariant and fully invariant subgroups of $G$.  The notion of $f$-inert subgroup allowed to introduce in \cite{GBS2,DGSV} two new notions of algebraic entropy: the \emph{intrinsic algebraic entropy} and the \emph{intrinsic adjoint algebraic entropy}.  

\medskip
In detail, let $(G,f)\in\Flow_{\Ab}$, where we denote by $\Ab$ the category of abelian groups and their homomorphisms.  Given an $f$-inert subgroup $H$ of $G$, the \emph{intrinsic algebraic entropy of $f$ with respect to $H$} is
\begin{equation}
\widetilde\ent(f,H)=\lim_{n\to\infty}\frac{1}{n}\log\card{\frac{H+f (H)+\cdots+f^{n-1}(H)}{H}},
\end{equation}
and the \emph{intrinsic algebraic entropy of $f$} is $\widetilde\ent(f)=\sup\{ \widetilde\ent(f,H)\mid H\in\mathcal I_f(G)\}.$ 

On the other hand, the \emph{intrinsic adjoint algebraic entropy of $f$ with respect to $H$} is
\begin{equation}
\widetilde\ent^*(f,H)=\lim_{n\to\infty}\frac{1}{n}\log\card{\frac{H}{H\cap f^{-1}(H)\cap\dots\cap f^{-n+1}(H)}},
\end{equation}
and so the \emph{intrinsic adjoint algebraic entropy of $f$} is $\widetilde\ent^*(f)=\sup\{ \widetilde\ent^*(f,H)\mid H\in\mathcal I_f(G)\}.$

\smallskip
Hereafter, we show that the intrinsic (respectively, intrinsic adjoint) algebraic entropy is part of the general scheme  introduced in this paper, namely, we prove them to be intrinsic functorial entropies with respect to suitable functors $\Flow_{\Ab}\to\Flow_{\barLqm}$. Recall that the family $\mathcal I_f(G)$ is a bounded sublattice of the lattice of all the subgroups of $G$ (see \cite[Lemma 2.6]{DGSV}).

\subsubsection{Intrinsic algebraic entropy for abelian groups}

For an abelian group $G$, denote $\mathcal S^{\vee}(G)=(\mathcal S(G), +, \subseteq)$, that is the family $\mathcal S(G)$ of all subgroups of $G$ partially ordered by inclusion and endowed with the ordinary sum  as join-operation; the zero element of $\mathcal S^{\vee}(G)$ is the trivial subgroup. By Example~\ref{ex:ab}(a),  $(\mathcal S^{\vee}(G),d_{[\,\colon]})\in{\barLqm}$.
In addition, for a morphism $f:G\to G'$ in $\Ab$, let 
$$\mathcal S^{\vee}(f)\colon (\mathcal S^{\vee}(G),d_{[\,\colon]})\to (\mathcal S^{\vee}(G'),d_{[\,\colon]}),$$ 
mapping $H\mapsto f(H).$
This defines the functor $\mathcal S^{\vee}: \Ab \to {\barLqm}$, which induces a functor $\overline{\mathcal S}^\vee:\Flow_\Ab \to \Flow_{\barLqm}$.

\begin{theorem}
On $\Flow_\Ab$, we have $\widetilde\ent=\hti_{\overline{\mathcal S}^\vee}$.
\end{theorem}
Indeed, $\mathcal I_{\mathcal S^{\vee}(f)}(\mathcal S^\vee(G),d_{[\,\colon]})=\mathcal I_f(G)$ and  $\widetilde\ent=\widetilde h\circ \overline{\mathcal S}^\vee$ (i.e., the following diagram commutes).
$$\xymatrix@-0.8pc{
 \Flow_{\Ab}\ar[dr]_{\widetilde\ent}\ar[rr]^{\overline{\mathcal S}^\vee}&      &\Flow_{\barLqm}\ar[dl]^{\widetilde h}\\
                                       &\R_{\geq0}\cup\{\infty\}&}$$

\subsubsection{Intrinsic adjoint algebraic entropy for abelian groups}

Conversely, for an abelian group $G$, let $\mathcal S^{\wedge}(G)=(\mathcal S(G),\cap,\supseteq)$ denote the family $\mathcal S(G)$ partially ordered by inverse inclusion together with the intersection of subgroups as join-operation. The semilattice $\mathcal S^{\wedge}(G)$ has $G$ as zero element. By Example~\ref{ex:ab}(b), $(\mathcal S^{\wedge}(G),d^*_{[\,\colon]})\in{\barLqm}$.
In addition, for a morphism $f:G\to G'$ in $\Ab$, let 
$$\mathcal S^{\wedge}(f)\colon (\mathcal S^{\wedge}(G'),d^*_{[\,\colon]})\to (\mathcal S^{\wedge}(G),d^*_{[\,\colon]}),$$
 mapping $H\mapsto f^{-1}(H).$
This defines the functor $\mathcal S^{\wedge}: \Ab \to {\barLqm}$, which induces a functor $\overline{\mathcal S}^\wedge:\Flow_\Ab \to \Flow_{\barLqm}$.

\begin{theorem}
On $\Flow_\Ab$, we have $\widetilde\ent^*=\hti_{\overline{\mathcal S}^\wedge}$.
\end{theorem}
Indeed, $\mathcal I_{\mathcal S^{\wedge}(f)}(\mathcal S^\wedge(G),d_{[\,\colon]})=\mathcal I_f(G)$ and $\widetilde\ent^*=\widetilde h\circ \overline{\mathcal S}^\wedge$ (i.e., the following diagram commutes).
$$\xymatrix@-0.8pc{
 \Flow_{\Ab}\ar[dr]_{\widetilde\ent^*}\ar[rr]^{\overline{\mathcal S}^\wedge}&      &\Flow_{\barLqm}\ar[dl]^{\widetilde h}\\
                                       &\R_{\geq0}\cup\{\infty\}&}$$

\subsubsection{Different choice of the semilattices}\label{diff:choice}

In order to obtain the intrinsic algebraic entropy and the intrinsic adjoint algebraic entropy as intrinsic functorial entropies, we can also proceed as follows.

\smallskip
For $(G,f)\in\Flow_\Ab$ 
let $\mathcal I^{\vee}_f(G)=(\mathcal I_f(G),d_{[\,\colon]})\in\barLqm$ be the subsemilattice of $\mathcal S^\vee(G)$ endowed with the generalized quasimetric induced by $d_{[\,\colon]}$. Moreover, let
$$\mathcal I^{\vee}_G(f)\colon (\mathcal I^{\vee}_f(G),d_{[\,\colon]})\to (\mathcal I^{\vee}_f(G),d_{[\,\colon]}),\ H\mapsto f(H).$$
Consequently, the assignment $(G,f)\mapsto ((\mathcal I^{\vee}_f(G),d_{[\,\colon]}),\mathcal I^{\vee}_G(f))$ produces the functor
$\mathcal I^{\vee}\colon\Flow_{\Ab}\to\Flow_{\barLqm}$, such that $\widetilde \ent=\widetilde h\circ \mathcal I^{\vee}$.
$$ \xymatrix@-0.8pc{
 \Flow_{\Ab}\ar[dr]_{\widetilde\ent}\ar[rr]^{\mathcal I^{\vee}}&      &\Flow_{\barLqm}\ar[dl]^{\widetilde h}\\
                                       &\R_{\geq0}\cup\{\infty\}&}$$

\smallskip
Analogously, let $\mathcal I^{\wedge}_f(G)=(\mathcal I_f(G),d^*_{[\,\colon]})\in\barLqm$ be the subsemilattice of $\mathcal S^\wedge(G)$ endowed with the generalized quasimetric induced by $d^*_{[\,\colon]}$.
Moreover, let
$$\mathcal I^{\wedge}_G(f)\colon (\mathcal I^{\wedge}_f(G),d_{[\,\colon]}^*)\to (\mathcal I^{\wedge}_f(G),d_{[\,\colon]}^*),\ H\mapsto f^{-1}(H).$$
This yields the functor
%
such that $\widetilde\ent^*=\hti_{\mathcal I^{\wedge}}$.
$$\xymatrix@-0.8pc{
 \Flow_{\Ab}\ar[dr]_{\widetilde\ent^*}\ar[rr]^{\mathcal I^{\wedge}}&      &\Flow_{\barLqm}\ar[dl]^{\widetilde h}\\
                                       &\R_{\geq0}\cup\{\infty\}&}$$

\subsection{Algebraic and topological entropy for locally compact groups} 

For a locally compact group $G$, denote by $\CO(G)$ the family of all compact open subgroups of $G$, which forms a neighborhood basis at $1_G$

\subsubsection{Algebraic entropy for compactly covered lca groups}

A topological group $G$ is said to be \emph{compactly covered} if each element of $G$ is contained in some compact subgroup of $G$. Let $\CC$ denote the category of compactly covered locally compact abelian groups and their continuous endomorphisms. For example, the additive group $\Q_p$ of $p$-adic rationals is an object of $\CC$. Compactly covered locally compact abelian groups are of great interest because they are the Pontryagin duals of totally disconnected locally compact abelian groups (see the next subsection). 

Let $(G,f)\in\Flow_{\CC}$. By \cite[Proposition~2.2]{DGB-BT}, 
the \emph{algebraic entropy of $f$ with respect to $U\in\mathcal{CO}(G)$} is
$$h_{alg}(f,U)=\lim_{n\to\infty}\frac{1}{n}\log[U+f(U)+\ldots+ f^{n-1}(U):U],$$
and $h_{alg}(f)=\sup\{ h_{alg}(f,U)\mid U\in\CO(G)\}$ is the \emph{algebraic entropy of $f$}.

For $G\in\CC$, we consider the semilattice $\CO^{\vee}(G)=(\CO(G)\cup\{0\},+,\subseteq)$ seen as a subsemilattice of $I_f(G)$ and so equipped with the generalized quasimetric $d_{[\,\colon]}$.
Then $(\CO^\vee(G),d_{[\,\colon]})\in\barLqm$.
Subsequently, for $f:G\to G'$ in $\CC$, let $\CO^\vee(f):\CO^\vee(G)\to\CO^\vee(G)$, $U\mapsto U+f(U)$.
This defines the functor $\CO^{\vee}\colon\CC\to\barLqm$, and so the functor $\overline{\CO}^{\vee}\colon\Flow_{\CC}\to\Flow_{\barLqm}$.

\begin{remark}
For every $(G,f)\in\Flow_{\CC}$ and every $U\in\CO(G)$, we always have $d_{[\,\colon]}(U,\CO^\vee(f)(U))=\log[U+f(U)\colon U]<\infty$, that is, 
\begin{equation}\label{I=tutto}
\CO^\vee(G)=\mathcal I_{\CO^\vee(f)}(\CO^\vee(G))\subseteq \mathcal I_f(G),
\end{equation}
so $\CO^{\vee}(G)$ is a subsemilattice of $\mathcal I_f(G)$.
\end{remark}

\begin{theorem}
On $\Flow_{\CC}$, we have $h_{alg}=\hti_{\overline{\CO}^\vee}$.
\end{theorem}

Indeed, the following diagram commutes by \eqref{I=tutto}.
$$
\xymatrix@-0.8pc{
\Flow_{\CC}\ar[dr]_{h_{alg}}\ar[rr]^{\overline{\CO}^{\vee}}&      &\Flow_{\barLqm}\ar[dl]^{\widetilde h}\\
                                       &\R_{\geq0}\cup\{\infty\}&} 
                                       $$

\subsubsection{Topological entropy for  t.d.l.c. groups} 

A locally compact group $G$ is said to be \emph{totally disconnected} if the connected component of the identity $1_G$ is reduced to the singleton $\{1_G\}$. Discrete groups and profinite groups are example of totally disconnected locally compact groups. In particular, profinite groups are precisely the topological groups that are compact and totally disconnected. 

Denote by $\TDLC$ the category of totally disconnected locally compact (= t.d.l.c.) groups and their continuous homomorphisms. 
As a consequence of van Dantzig's theorem, $\CO(G)$ is a neighborhood basis at $1_G$ whenever $G\in\TDLC$. As pointed out in \cite{DG-islam,DSV}, such a property allows to define the topological entropy of continuous endomorphisms of $G$ without resorting to the Haar measure, as follows. 

Let $(G,f)\in\Flow_{\TDLC}$. The \emph{topological entropy of $f$ with respect to $U\in \CO(G)$} is 
$$h_{top}(f,U)=\lim_{n\to\infty}\frac{1}{n}\log[U\colon U\cap f^{-1}(U)\cap\dots\cap f^{-n+1}(U)],$$
and $h_{top}(f)=\sup\{ h_{top}(f,U)\mid U\in\CO(G)\}$ denotes the \emph{topological entropy of $f$}. 

\medskip
For $G\in\TDLC$ we consider the semilattice $\CO^{\wedge}(G)=(\CO(G)\cup\{G\},\cap,\supseteq)$ equipped with the generalized quasimetric $d^*_{[\,\colon]}$. 
Therefore,  $(\CO^\wedge(G),d^*_{[\,\colon]})\in\barLqm$.
Subsequently, for $f:G\to G'$ in $\TDLC$, let $\CO^\wedge(f):\CO^\wedge(G')\to\CO^\wedge(G)$, $U\cap f^{-1}(U)$.
This defines a functor $\CO^\wedge:\TDLC\to \barLqm$, which induces a functor $\overline{\CO}^{\wedge}\colon\Flow_{\TDLC}\to\Flow_{\barLqm}$.

\begin{remark}
For every $(G,f)\in\Flow_\TDLC$ and every $U\in\CO(G)$, we always have $d^*_{[\,\colon]}(U,\CO^\wedge(f)(U))=\log[U\colon U\cap f^{-1}(U)]<\infty$, that is, 
\begin{equation}\label{I=tutto'}
\CO^\wedge(G)=\mathcal I_{\CO^\wedge(f)}(\CO^\wedge(G))\subseteq \mathcal I_f(G),
\end{equation}
and in particular $\CO^{\wedge}(G)$ is a subsemilattice of $\mathcal I_f(G)$.
\end{remark}

\begin{theorem}
On $\Flow_\TDLC$, we have $h_{top}=\hti_{\overline{\CO}^\wedge}$.
\end{theorem}

Indeed, the following diagram commutes by \eqref{I=tutto'}.
$$\xymatrix@-0.8pc{
 \Flow_{\TDLC}\ar[dr]_{h_{top}}\ar[rr]^{\overline{\CO}^{\wedge}}&      &\Flow_{\barLqm}\ar[dl]^{\widetilde h}\\
                                       &\R_{\geq0}\cup\{\infty\}&}$$

\subsection{Algebraic and topological entropy for l.l.c.\! vector spaces}\label{ss:llc}

\subsubsection{Locally linearly compact vector spaces} 

Let $\K$ be a discrete field. A topological $\K$-vector space $V$ is {\it linearly compact} when:
\begin{enumerate}[(LC1)]
\item it is a Hausdorff space in which there is a neighborhood basis at $0$ consisting of linear subspaces of $V$;
\item any collection of closed linear varieties (i.e., closed cosets of linear subspaces) of $V$ with the finite intersection property has non-empty intersection.
\end{enumerate}
For example, finite-dimensional discrete vector spaces are linearly compact, and  compact vector spaces satisfying (LC1) are linearly compact. More precisely, every linearly compact $\K$-space is a Tychonoff product of one-dimensional $\K$-spaces, and viceversa. 
Let $\LC$ denote the category of linearly compact $\K$-vector spaces and their continuous homomorphisms. We collect here a few properties of linearly compact vector spaces (see \cite{Lef}) that we use further on. Let $W\leq V$, $U$ be  $\K$-vector spaces  satisfying condition (LC1), thus:
\begin{enumerate}[label=(lc\arabic{enumi}), ref=(lc\arabic{enumi})]
\item\label{lc:homo} if $\phi\colon V\to U$ is a surjective continuous homomorphism and $V$ is linearly compact, then $U$ is linearly compact;
\item\label{lc:closed} if $V$ is linearly compact and $W$ is closed, then $W$ is linearly compact;
\item\label{lc:main} if $V$ is discrete, then $V$ is linearly compact if and only if $V$ has finite dimension over $\K$;
\item\label{lc:ses} if $W$ is closed, then $V$ is linearly compact if and only if $W$ and $V/W$ are linearly compact.
\end{enumerate}

A topological $\K$-vector space $V$ is said to be  \emph{locally linearly compact} if the family $\LCO(V)$ of all linearly compact open linear subspaces of $V$ is a neighborhood basis at $0$. Let $\LLC$ denote the category of locally linearly compact $\K$-vector spaces and their continuous homomorphisms. 
The category $\LC$ is a full subcategory of $\LLC$, and also the category $\Vect$ of discrete $\K$-vector spaces is a full subcategory of $\LLC$.

\begin{remark}
The partially ordered set $(\LCO(V),\subseteq)$ is a lattice with join-operation given by the sum of linear subspaces (see~\ref{lc:homo}) and  meet-operation  given by the intersection (see~\ref{lc:closed}). 
The lattice  $(\LCO(V),\subseteq)$ is not bounded unless $V$ has finite dimension. If $V$ is discrete, then  $(\LCO(V),\subseteq,+)$ has as zero element $0$. If $V$ is linearly compact, then $(\LCO(V),\supseteq,\cap)$ has as zero element $V$. 
\end{remark}

\subsubsection{Algebraic Entropy for locally linearly compact vector spaces}

Following \cite{CGBalg}, for every flow $(V,f)$ over $\LLC$, the \emph{algebraic entropy of $f$ with respect to $U\in\LCO(V)$} is
$$\ent(f,U)=\lim_{n\to\infty}\frac{1}{n}\dim\frac{U+f (U)+\ldots+f^{n-1}(U)}{U},$$
and the \emph{algebraic entropy of $f$} is $\ent(f)=\sup\{ \ent(f,U)\mid U\in\LCO(V)\}.$

For $V\in\LLC$, let $\LCO^{\vee}(V)$ denote the semilattice $(\LCO(V)\cup\{0\},\subseteq,+)$ with zero element given by the trivial subspace. Recall that the trivial subspace of $V$ is not open unless $V$ is discrete, and therefore we need to add it. By Example~\ref{ex:inv}, $\LCO^{\vee}(V)$ inherits the generalized quasimetric $d_{\dim}$. Then $(\LCO^{\vee}(V),d_{\dim})\in\barLqm$.

Moreover, for a morphism $f:V\to V'$ in $\LLC$, let $\LCO^{\vee}(f):\LCO^{\vee}(V)\to \LCO^{\vee}(V')$, $U\mapsto U+f(U)$.
This gives us the functor $\LCO^{\vee}\colon{\LLC}\to{\barLqm}$, which induces the functor $\overline{\LCO}^{\vee}\colon\Flow_{\LLC}\to\Flow_{\barLqm}$.

\begin{remark}
For every $(V,f)\in\Flow_{\LLC}$ and every $U\in\LCO(V)$ we always have $d_{\dim}(U,\LCO^\vee(f)(U))=\dim(U+f(U)/U)<\infty$ by~\ref{lc:ses} and~\ref{lc:main}, that is, 
\begin{equation}\label{I=tuttodim}
\LCO^\vee(V)=\mathcal I_{\LCO^\vee(f)}(\LCO^\vee(V))\subseteq \mathcal I_f(V),
\end{equation}
and in particular $\LCO^{\vee}(V)$ is a subsemilattice of $\mathcal I_f(V)$.
\end{remark}

\begin{theorem}\label{ent=lco}
On $\Flow_{\LLC}$, we have $\ent=\hti_{\overline{\LCO}^\vee}$.
\end{theorem}

Indeed, in view of \eqref{I=tuttodim} the following diagram commutes.
$$\xymatrix@-0.8pc{
 \Flow_{\LLC}\ar[dr]_{\ent}\ar[rr]^{\overline{\LCO}^\vee}&      &\Flow_{\barLqm}\ar[dl]^{\widetilde h}\\
                                       &\R_{\geq0}\cup\{\infty\}&}$$

%
%
%

\subsubsection{Topological Entropy for locally linearly compact vector spaces}

The topological counterpart of the algebraic entropy  for locally linearly compact vector spaces was introduced in \cite{CGBtop} as follows. The \emph{topological entropy of $f$ with respect to $U\in\LCO(V)$} is
$$\ent^*(f,U)=\lim_{n\to\infty}\frac{1}{n}\dim\frac{U}{U\cap f^{-1}(U)+\ldots+f^{-n+1}(U)},$$
and the \emph{topological entropy of $f$} is $\ent^*(f)=\sup\{ \ent^*(f,U)\mid U\in\LCO(V)\}$.

For $V\in{\LLC}$ consider the semilattice $\LCO^{\wedge}(V)$ given by $(\LCO(V)\cup\{V\},\supseteq,\cap)$; the semilattice  $\LCO^{\wedge}(V)$  has zero element $V$. We consider on $\LCO^{\wedge}(V)$ the generalized quasimetric $d^*_{\dim}$ from Example~\ref{ex:inv}.
Then $(\LCO^{\wedge}(V),d_{\dim})\in\barLqm$.
Moreover, for a morphism $f:V\to V'$ in $\LLC$, let $\LCO^{\wedge}(f):\LCO^{\wedge}(V')\to \LCO^{\wedge}(V)$, $U\mapsto U\cap f^{-1}(U)$.
This produces the functor $\LCO^{\wedge}\colon{\LLC}\to{\barLqm}$, which induces the functor $\overline{\LCO}^{\wedge}\colon\Flow_{\LLC}\to\Flow_{\barLqm}$.

\begin{remark}
For $(V,f)\!\in\!\Flow_{\LLC}$ and  $U\!\in\!\LCO(V)$, by~\ref{lc:closed},~\ref{lc:main} and~\ref{lc:ses}, 
$d^*_{\dim}(U,\LCO^\wedge(f)(U))=\dim(U/U\cap f^{-1}(U))<\infty$ , that is, 
\begin{equation}\label{I=tuttodimagain}
\LCO^\wedge(V)=\mathcal I_{\LCO^\wedge(f)}(\LCO^\wedge(V))\subseteq \mathcal I_f(V),
\end{equation}
and in particular $\LCO^{\wedge}(V)$ is a subsemilattice of $\mathcal I_f(V)$.
\end{remark}

\begin{theorem}
On $\Flow_{\LLC}$, we have $\ent^*=\hti_{\overline{\LCO}^\wedge}$.
\end{theorem}
Indeed, by \eqref{I=tuttodimagain}  the following diagram commutes.
\begin{equation}
 \xymatrix@-0.8pc{
 \Flow_{\LLC}\ar[dr]_{\ent^*}\ar[rr]^{\overline{\LCO}^{\wedge}}&      &\Flow_{\barLqm}\ar[dl]^{\widetilde h}\\
                                       &\R_{\geq0}\cup\{\infty\}&}
\end{equation}

\end{document}